\documentclass{article}

\usepackage{fullpage}
\usepackage[colorlinks,linkcolor=blue,citecolor=blue]{hyperref}

\usepackage{amsmath}
\usepackage{amsthm}
\newtheorem{theorem}{Theorem}
\newtheorem{lemma}[theorem]{Lemma}
\newtheorem{corollary}[theorem]{Corollary}
\newtheorem{conjecture}[theorem]{Conjecture}
\theoremstyle{definition}
\newtheorem{remark}[theorem]{Remark}

\renewcommand{\d}[1][G]{\overline{d}(#1)}
\newcommand{\lspectrum}[2][\mu]{#1_1\geq #1_2\geq\cdots\geq #1_{#2}}

\title{Partial characterization of graphs having a single large Laplacian eigenvalue}

\author{L.~Emilio Allem\footnote{Instituto de Matem\'atica, Universidade Federal do Rio Grande do Sul, Brazil. E-mail address: \texttt{emilio.allem@ufrgs.br}} \and Antonio Cafure\footnote{CONICET, Instituto del Desarrollo Humano, Universidad Nacional de General Sarmiento, and Departamento de Matem{\'a}tica, CBC, Universidad de Buenos Aires, Argentina. E-mail address: \texttt{acafure@ungs.edu.ar}} \and Ezequiel Dratman\footnote{CONICET and Instituto de Ciencias, Universidad Nacional de General Sarmiento, Argentina. E-mail address: \texttt{edratman@ungs.edu.ar}} \and Luciano N.~Grippo\footnote{Instituto de Ciencias, Universidad Nacional de General Sarmiento, Argentina. E-mail address: \texttt{lgrippo@ungs.edu.ar}} \and Mart\'in D.~Safe\footnote{Departamento de Matem\'atica, Universidad Nacional del Sur, Argentina. E-mail address: \texttt{msafe@uns.edu.ar}} \and Vilmar Trevisan\footnote{Instituto de Matem\'atica, Universidade Federal do Rio Grande do Sul, Brazil. E-mail address: \texttt{trevisan@mat.ufrgs.br}}}

\begin{document}

\maketitle

\begin{abstract}
   The parameter $\sigma(G)$ of a graph $G$ stands for the number of Laplacian eigenvalues greater than or equal to the average degree of $G$. In this work, we 
address the problem of characterizing those graphs $G$ having $\sigma(G)=1$. Our conjecture is that these graphs are stars plus a (possible empty) set of 
isolated vertices. We establish a link between $\sigma(G)$ and the number of anticomponents of $G$. As a by-product, we present some results which support the 
conjecture, by restricting our analysis to some classes of graphs.
\end{abstract}

\section{Introduction}
 Let $G$ be  a graph on $n$ vertices and $m$ edges and let $d_1 \geq \cdots \geq d_n$ be its degree sequence. We denote by  $A(G)$  its adjacency matrix and by  $D(G)$ the diagonal matrix having $d_i$ in the diagonal entry $(i,i)$, for every $1\le i\le n$,  and $0$ otherwise. The \emph{Laplacian matrix} of $G$ is the positive semidefinite matrix $L(G)=D(G)-A(G)$. The  eigenvalues of $L(G)$ are called \emph{Laplacian eigenvalues} of $G$;  the spectrum of $L(G)$  is the   \emph{Laplacian spectrum} of $G$ and will be denoted by $Lspec(G)$.  Since it is  easily seen that $0$ is a Laplacian eigenvalue  and it is well-known that Laplacian eigenvalues are less than or equal to  $n$ it turns out that $Lspec(G) \subset [0,n]$. From now on, if $Lspec(G) =\{\mu_1,\mu_2,\ldots,\mu_n\}$, we will assume that $\mu_1 \geq \mu_2 \geq \cdots \geq \mu_n$, where $\mu_n=0$.

 Understanding the distribution of Laplacian eigenvalues of graphs is a problem that is both relevant and difficult. It is relevant due to the many applications related  to Laplacian matrices (see, for example \cite{Moh91,Moh92}). It seems to be difficult because little is known about how the Laplacian eigenvalues are distributed in the interval $[0,n]$.

 Our main motivation is understanding the structure of graphs that have few large Laplacian eigenvalues. In particular, we would like to characterize graphs that have a single large Laplacian eigenvalue. What do we mean by a large Laplacian eigenvalue?  A reasonable measure is to compare this
 eigenvalue with the average of all eigenvalues. Since the average of Laplacian eigenvalues equals the average degree $\d = \frac{2m}{n}$ of $G$, we say that a Laplacian eigenvalue is \emph{large} if it is greater than or equal to the average degree.

 Inspired by this idea, the paper~\cite{Das16} introduces the spectral parameter $\sigma(G)$ which counts
 the number of Laplacian eigenvalues greater than or equal to $\d$. Equivalently,  $\sigma(G)$ is the largest index $i$ for which $\mu_i \geq \frac{2m}{n}$. Since  the greatest Laplacian eigenvalue $\mu_1$ is at least $\frac{2m}{n}$ then it follows that  $\sigma(G)\ge 1$. 

 There is evidence that $\sigma(G)$ plays an important role in defining structural properties of a graph $G$. For example, it is related to the clique number $\omega$ of $G$ (the number of vertices of the largest induced  complete subgraph of $G$) and it also gives insight about the Laplacian energy of a graph \cite{Pirzada2015,Das16}. Moreover, several structural properties of a graph are related to  $\sigma$ (see, for example \cite{Das2015,Das16}).

 In this paper we are concerned with furthering the study of  $\sigma(G)$. In particular, we deal with a  problem posed in  \cite{Das16} which asks for characterizing all graphs $G$ having $\sigma(G)=1$; \emph{i.e.}, having only one large Laplacian eigenvalue. Our conjecture is that the only connected graph on $n$ vertices having $\sigma=1$ is the star $K_{1,n-1}$ and that the  only nonconnected graph on  $n$ vertices having $\sigma=1$ is a star together with some isolated vertices. More precisely, 
 we conjecture that  graphs having $\sigma = 1$ are some stars plus a (possibly empty) set of isolated vertices. From now on, $K_{1,r}+sK_1$ denotes the star on $r+1$ vertices plus $s$ isolated vertices.

\begin{conjecture}\label{conjecture:1}
 Let $G$ be a graph. Then $\sigma(G)=1$ if and only if $G$ is isomorphic to $K_1$, $K_2+sK_1$ for some $s\geq 0$, or $K_{1,r}+sK_1$ for some $r\geq 2$ and $0\leq s<r-1$.
\end{conjecture}

In this work, we show that this conjecture is true if it holds for graphs which are simultaneously connected and co-connected (Conjecture~\ref{conjecture:3}) and prove that Conjecture~\ref{conjecture:1} is true for forests and extended $P_4$-laden graphs~\cite{MR1420325} (a common superclass of split graphs and cographs). The main tool for  proving our results is an interesting  link we have found between  $\sigma$ and the number of  anticomponents of $G$ (see Section~\ref{section:definitions}). %
The interesting feature of this result is that it relates a spectral parameter with a classical structural parameter. Studying structural properties of the anticomponents of $G$ may shed light on the distribution of Laplacian eigenvalues and, reciprocally, the distribution of Laplacian eigenvalues should give insight about the structure of the graph.

This article is organized as follows. In Section~\ref{section:definitions} we state  definitions and previous results concerning Laplacian eigenvalues. In Section~\ref{sec:anticomponents}, we present some new results which establish the connection between $\sigma$ and the number of nonempty anticomponents of $G$. In Section~\ref{section: sigma equals one}, we present some evidence on the validity of Conjecture~\ref{conjecture:1} by proving that the conjecture is true when $G$ is either a forest,  or a $P_4$-laden graph.

\section{Definitions}\label{section:definitions}

In this article, all graphs are finite, undirected, and without multiple edges or loops. All definitions and concepts not introduced here can be found in~\cite{west_introduction_2000}. We say that a graph is \emph{empty} if it has no edges. A \emph{trivial} graph is a graph with  precisely one vertex; every trivial graph is isomorphic to the graph which we will denote by $K_1$. A graph is \emph{nontrivial} if it has more than one vertex.

We use the standard notation $\Delta(G)$ to denote the maximum degree of a graph $G$.  

Let $G_1$ and $G_2$ be two graphs such that $V(G_1)\cap V(G_2)=\emptyset$. The \emph{disjoint union} of $G_1$ and $G_2$, denoted $G_1+G_2$, is the graph whose vertex set is $V(G_1)\cup V(G_2)$, and its edge set is $E(G_1)\cup E(G_2)$. We write $kG$ to represent the disjoint union $G + \cdots + G$ of $k$ copies of a graph $G$. 
The \emph{join} of $G_1$ and $G_2$, denoted $G_1 \vee G_2$, is the graph obtained from $G_1 + G_2$ by adding new edges from each vertex  of $G_1$ to every vertex of $G_2$. 

A vertex $v$ of a graph $G$ is a \emph{twin} of another vertex $w$ of $G$ if they both have the same neighbors in $V(G)\setminus\{v,w\}$. We say that a graph $G'$ is obtained from $G$ \emph{by adding a twin} $v^\prime$ to a vertex $v$ of $G$ if $V(G')=V(G)\cup\{v'\}$, $v'$ is a twin of $v$ in $G'$,  and $G'-v'$ is isomorphic to $G$. 

By $G[S]$ we denote the subgraph  of $G$ induced by a subset $S \subseteq V(G)$. 

We use $\overline G$ to denote the complement graph of a graph $G$. An \emph{anticomponent} of a graph $G$ is the subgraph of $G$ induced by the vertex set of a connected component of $\overline G$. More precisely, an induced subgraph $H$ of $G$ is an anticomponent if $\overline H$ is a connected component of $\overline{G}$. Notice that if $G_1,G_2,\ldots,G_k$ are the anticomponents of $G$, then $G=G_1\vee\cdots\vee G_k$. A graph $G$ is \emph{co-connected} if $\overline{G}$ is connected.

A \emph{forest} is a graph with no cycles and a \emph{tree} is a connected forest. The complete graph on $n$ vertices is denoted by $K_n$. A \emph{universal vertex} of a graph $G$ is a vertex $v$ adjacent to every vertex $w$ different from $v$. A \emph{star} is a graph   isomorphic to $K_1$ or to a tree with a universal vertex. We use $K_{1,n-1}$ to denote the star on $n$ vertices, where $K_{1,0}$ is isomorphic to $K_1$ and $K_{1,1}$ is isomorphic to $K_2$. The \emph{chordless path} (respectively,  \emph{cycle}) on $k$ vertices is denoted by $P_k$ (respectively, \ $C_k$). 

A \emph{stable set} of a graph is a set of pairwise nonadjacent vertices. A \emph{clique} of a graph is a set of pairwise adjacent vertices. 

Throughout this article, given  two graphs $G$ and $H$,   we write $G=H$ to point out  that $G$ and $H$ belong to the same isomorphism class.

The following well-known result provides a lower bound for the largest Laplacian eigenvalue of a graph with at least one edge in terms of the maximum degree of the graph.

\begin{lemma} [\cite{Gro90}] \label{lemma:cotainfmu1}
 Let $G$ be a graph on $n$ vertices with at least one edge. Then $\mu_1(G) \geq 1 + \Delta(G)$.
\end{lemma}

The second largest Laplacian eigenvalue of a graph is lower bounded by the second term of the degree sequence of the graph.

\begin{lemma}[\cite{MR1813439}]\label{lemma:second largest eigenvalue}
 Let  $G$ be a graph with degree sequence $\lspectrum[d]{n}$ and spectrum $\lspectrum{n}=0$.  Then $\mu_2\ge d_2$.
\end{lemma}

It is worth mentioning that Brouwer and Haemers~\cite{BrowerandHamers2008} generalized the above result by presenting a lower bound for the \emph{k}th greatest Laplacian eigenvalue in terms of $d_k$, answering a conjecture raised by Guo~\cite{Guo2007}.

It is easy to prove  that the Laplacian spectrum of the disjoint union $G_1 + G_2$ is the union of the Laplacian spectrums of $G_1$ and $G_2$. The next result  allows to determine the Laplacian spectrum of the join $G_1 \vee G_2$,  from those of $G_1$ and $G_2$.

\begin{theorem}[{\cite[Theorem 2.20]{MR1275613}}]\label{theorem:laplacian spectrum  disjoint union and join}
 Let $G_1$ and $G_2$ be two graphs with Laplacian spectrums $\lspectrum{n_1}=0$ and $\lspectrum[\lambda]{n_2}=0$, respectively. %
 Then the Laplacian eigenvalues of  $G_1\vee G_2$ are $n_1 + n_2$;  $n_2 + \mu_i$, for $1\leq i \leq n_1-1$; $n_1 + \lambda_{i}$, for $1\leq i \leq n_2-1$ and 0.
\end{theorem}

\section{Relating $\sigma$ and the number of anticomponents}\label{sec:anticomponents}

This section is devoted to establish a link between $\sigma(G)$ and the number of anticomponents of $G$.

In virtue of Theorem~\ref{theorem:laplacian spectrum  disjoint union and join}, the following result immediately holds.

 \begin{lemma}\label{lemma:multjoin}
 If $G=G_1\vee\cdots\vee G_k$, with  $k \geq 1$, is a graph on $n$ vertices, then $n$ is a Laplacian  eigenvalue of $G$
 with multiplicity at least $k-1$.
\end{lemma}

 \begin{lemma}\label{lemma:anticomponents less than or equal to  sigma + 1}
 If $G$ has  $k$ anticomponents,  then $k \leq \sigma(G) + 1$.
\end{lemma}

 \begin{proof}
 Let $G=G_1\vee\cdots\vee G_k$ where $G_1,\ldots,G_k$ are the anticomponents of $G$.
 For any graph $G$ with at least one vertex we have that  $\sigma(G) \geq 1$ and thus the assertion follows when $k=1$. We may assume that $k \geq 2$. Lemma~\ref{lemma:multjoin}  implies that $n$ is a Laplacian  eigenvalue of $G$ with  multiplicity at least $k-1$ in $G$. Thus $\mu_{k-1}(G) =  n$ which implies that   $\sigma(G) \geq k-1$.
 \end{proof}

\begin{remark}
 The upper bound given by Lemma \ref{lemma:anticomponents less than or equal to  sigma + 1} is sharp  when  $\sigma(G) > 1$. 
 Indeed, for  $s\geq 2$ consider  the graph $G = 4 K_2 \vee K_1\vee \cdots \vee K_1$, where $s$ is the number of $K_1$'s. The  average degree of $G$ 
  is $s+7 - \frac{48}{s+8}$ and it has $s + 1$ anticomponents.  Since its Laplacian eigenvalues are  $s+8$, $s+2$, $s$, and $0$ with multiplicities $s$, $4$, $3$, and $1$, respectively, it follows that  $\sigma(G) = s$.
\end{remark}

 We use $\ell(G)$ to denote the number of nonempty anticomponents of a graph $G$. Recall that a nontrivial graph has at least two vertices. The following result looks further into the case where equality holds in  Lemma~\ref{lemma:anticomponents less than or equal to  sigma + 1} showing  that $\sigma(G)$ is an upper bound for $\ell(G)$.

\begin{theorem}\label{theorem:anticomponents}
 Let $G$ be a graph having $k = \sigma(G)+1$ anticomponents. Then $\ell(G) \leq \sigma(G)$. Moreover, if $\sigma(G)=\ell(G)$, then the remaining anticomponent of $G$ is empty but nontrivial.
\end{theorem}

\begin{proof}
 Write  $G=G_1\vee\cdots\vee G_k$ where $G_1,\ldots,G_k$ are the anticomponents of $G$. Since $\sigma(G)\ge 1$ then $k\ge 2$. We set the following notations  for each $i\in\{1,\ldots,k\}$: \[
    n_i=\vert V(G_i)\vert, \quad  m_i=\vert E(G_i)\vert, \quad  \mu^{(i)}_1=\mu_1(G_i).
 \]
 Assume that $G_1, \ldots, G_\ell$ are the nonempty anticomponents. Since $k\geq 2$ and we are assuming that $\sigma(G) = k-1$ it turns out that  $\mu_k(G) < \overline d(G)$. Therefore, for each $i\in\{1,\ldots,k\}$ such that $n_i>1$ we have that
 \[
   n-n_i+\mu_1^{(i)}\leq\mu_k(G)< \frac{2m}n=\frac{2\sum_{j=1}^km_j+2\sum_{1\leq i<j\leq k}n_in_j}n,
 \]
 the first inequality holds by Theorem~\ref{theorem:laplacian spectrum  disjoint union and join}. Equivalently,
 \begin{equation}
 \begin{aligned}\label{eq:mu1}
  \mu_1^{(i)}&<\frac{2\sum_{j=1}^km_j-(n^2-2\sum_{1\leq i<j\leq k}n_in_j)}n+n_i\\
             &=\frac{2\sum_{j=1}^km_j-\sum_{j=1}^k n_j^2+nn_i}n.
 \end{aligned}
 \end{equation}

 As a consequence of Lemma~\ref{lemma:cotainfmu1}, we
 obtain  the following  lower bound for each  $i\in\{1,\ldots,\ell\}$:
 \begin{equation}\label{eq:mu2}
  \mu_1^{(i)}\geq\Delta(G_i)+1\geq\overline d(G_i)+1=\frac{2m_i}{n_i}+1.
 \end{equation}
 Combining  \eqref{eq:mu1} and \eqref{eq:mu2}, we deduce that, for each $i\in\{1,\ldots,\ell\}$,
 \begin{equation}\label{eq:mu3}
  2n_i\sum_{j=1}^k m_j-n_i\sum_{j=1}^k n_j^2+nn_i^2-2nm_i-nn_i>0.
 \end{equation}

Arguing towards a contradiction, suppose that $\ell(G) = k$. If we sum up the left-hand side of \eqref{eq:mu3} for each $i\in\{1,\ldots,k\}$, we obtain
 \[
   2n\sum_{j=1}^km_j-n\sum_{j=1}^kn_j^2+n\sum_{i=1}^kn_i^2-2n\sum_{i=1}^km_i-n^2 = -n^2
 \]
 which is not a positive quantity. This contradiction proves that $G$ has at most $k-1 = \sigma(G)$ nonempty anticomponents and our first assertion follows.

Assume now that $\ell(G)= k-1$. Suppose that $G_k$ is trivial. Hence $n_k=1$ and $m_k=0$. Summing up to the left-hand side of \eqref{eq:mu3} for each $i\in\{1,\ldots,k-1\}$, we obtain that
 \[
   -2\sum_{j=1}^{k-1}m_j+\sum_{j=1}^k n_j^2-n^2 = -2\sum_{j=1}^{k-1} m_j-2\sum_{1\leq i<j\leq k}n_in_j
 \]
 should be a positive number. This contradiction shows that $G_k$ must be nontrivial.
 \end{proof}

Recall that a \emph{bipartite graph} is a graph whose set of vertices can be partitioned into two (possibly empty) stable sets called \emph{partite sets} of the bipartite graph. A \emph{complete bipartite graph} is a bipartite graph isomorphic to $rK_1\vee sK_1$ for two positive integers $r$ and $s$.  We  denote by $K_{r,s}$  the complete bipartite graph isomorphic to $rK_1\vee sK_1$. The upper bound $\sigma(G)$ on $\ell(G)$ for those graphs having exactly $\sigma(G)+1$ anticomponents is not tight when $\sigma(G)=1$. Indeed, the following result shows that if a graph $G$ has $\sigma(G)=1$, then $G$ has no nonempty anticomponents.

\begin{corollary}\label{corollary:co-connected}
 If $G$ is a graph with $\sigma(G)=1$ and $\overline{G}$ is disconnected, then $G$ is a complete bipartite graph.
\end{corollary}

 \begin{proof}
In virtue of Lemma~\ref{lemma:anticomponents less than or equal to  sigma + 1}, the number of anticomponents of $G$ is at most $2$. Since $\overline{G}$ is  disconnected, we conclude that $G$ has precisely two anticomponents $G_1$ and $G_2$ and thus $G = G_1 \vee G_2$.

Suppose, for a contradiction, that $G_1$ is a  nonempty anticomponent of $G$. Because of Theorem~\ref{theorem:anticomponents}, we conclude that
$G_2$ is empty but nontrivial. Following the notation used in the proof of Theorem~\ref{theorem:anticomponents}, we have that $m_2=0$. For $i=1$,   inequality~\eqref{eq:mu3} becomes
 \begin{equation}\label{eq:mu4}
  -2n_2m_1-n_1n_2^2+n_2n_1^2-n_1^2-n_1n_2 > 0.
 \end{equation}
Since $G_2$ is a nontrivial empty graph it follows that  $\mu_1^{(2)}=0$ and hence,  for $i=2$,  inequality~\eqref{eq:mu1}  becomes
 \begin{equation}\label{eq:mu5}
  2m_1-n_1^2 + n_1n_2 >0.
 \end{equation}
 Summing up \eqref{eq:mu4} and  $n_2$ times \eqref{eq:mu5} gives
 \[
  -n_1^2-n_1n_2>0.
 \]
This contradiction arose from supposing that $G$ has some nonempty anticomponent. Hence, both anticomponents of $G$ are empty; \emph{i.e.}, $G$ is a complete bipartite graph.
 \end{proof}

\section{Graphs with $\sigma = 1$}\label{section: sigma equals one}

In this section we provide some evidence in order to make plausible Conjecture~\ref{conjecture:1}. We first verify Conjecture~\ref{conjecture:1} for graphs having disconnected complement; namely, we prove that the only graphs having $\sigma=1$ and disconnected complement are the stars (including the trivial star $K_1$). Then, we prove that Conjecture~\ref{conjecture:1} can be reduced to proving that the only connected and co-connected graph with $\sigma=1$ is $K_1$. We then verify Conjecture~\ref{conjecture:1} for extended $P_4$-laden graphs, a common superclass of the classes of cographs and split graphs.

\subsection{Reduction to co-connected graphs}

We first obtain a result which proves the validity of Conjecture~\ref{conjecture:1} for graphs having disconnected complement.

\begin{theorem}\label{theorem:disconnected complement}
 Let $G$ be a graph on $n$ vertices such that $\overline{G}$ is disconnected. Then $\sigma(G) = 1$ if and only if $G = K_{1,n-1}$.

\end{theorem}

 \begin{proof}
  Assume first that $G = K_{1,n-1}$. Then $\d=2-\frac{2}{n}$.  If $n = 2$, the Laplacian eigenvalues of $G$ are $2$ and $0$. If $n\geq 3$, the Laplacian eigenvalues of $G$ are $n$, $1$ and 0, each with multiplicity $1$, $n-2$ and $1$, respectively.  In any case we have that $\sigma(G) = 1$.

  Conversely, assume that $\sigma(G) = 1$. Corollary~\ref{corollary:co-connected} implies that  $G=K_{n_1,n_2}$, where $n_2\geq n_1\geq 1$ and $n = n_1+n_2$. The average degree of $G$ is equal to $\frac{2n_1n_2}n$.   
  In virtue of Theorem~\ref{theorem:laplacian spectrum  disjoint union and join}, the Laplacian eigenvalues  of $K_{n_1,n_2}$ are 
  $n$, $n_2$, $n_1$ and 0, each with multiplicity $1$, $n_1-1$, $n_2-1$ and $1$, respectively.
  
  Arguing towards a contradiction, suppose that $n_1\geq 2$. Hence $\mu_2(G) = n_2$. Since  $ 2n_1\le n$ we deduce that $\overline d(G)=\frac{2n_1n_2}n \leq \mu_2(G)$, which contradicts the fact that  $\sigma(G)=1$. This contradiction proves that $n_1=1$ and therefore we conclude that  $G=K_{1,n-1}$.
 \end{proof}

As a consequence of Theorem~\ref{theorem:disconnected complement}, Conjecture~\ref{conjecture:1} is equivalent to the validity of the following weaker conjecture.

\begin{conjecture}\label{conjecture:2}
 Let $G$ be a graph with connected complement. Then, $\sigma(G)=1$ if and only if $G$ is isomorphic to $K_1$, $K_2+sK_1$ for some $s> 0$, or $K_{1,r}+sK_1$ for some $r\geq 2$ and $0< s<r-1$.
\end{conjecture}

\subsection{Reduction to connected and co-connected graphs}
 We next show that the validity of Conjectures~\ref{conjecture:1} and~\ref{conjecture:2} can be reduced to the validity of the following even weaker conjecture.

\begin{conjecture}\label{conjecture:3}
 Let $G$ be a connected graph with connected complement. Then, $\sigma(G)=1$ if and only if $G$ is isomorphic to $K_1$.
\end{conjecture}

A graph class $\mathcal G$ is \emph{closed by taking components} if every connected component of every graph in $\mathcal G$ also belongs to $\mathcal G$. In particular, the class of all graphs is closed by taking components. Below we prove that the reduction from Conjecture~\ref{conjecture:1} to Conjecture~\ref{conjecture:3} holds even when restricted to any graph class closed by
taking components.

\begin{theorem}\label{theorem:+}
 Let $\mathcal G$ be a graph class closed by taking components. If Conjecture~\ref{conjecture:3} holds for $\mathcal G$, then
Conjecture~\ref{conjecture:1} also holds for $\mathcal G$.
\end{theorem}

 \begin{proof}
 Let $G$ be a graph in $\mathcal G$ with $\sigma(G)=1$. Assume first that $G$ is connected. If $G$ is co-connected, by hypothesis, $G$ is isomorphic to $K_1$. If $G$ is not co-connected, then $G$ is isomorphic to $K_{1,r}$ for some $r\geq 1$,  by virtue of Theorem~\ref{theorem:disconnected complement}.

Assume now that $G$ is disconnected and let $G=G_1+G_2$, where each of $G_1$ and $G_2$ has at least one vertex. We can assume, without loss of generality, that $G_1$ is connected and $\mu_1(G_1)\geq\mu_1(G_2)$.  If $G_1$ were empty, then $G_2$ would also be empty, contradicting $\sigma(G)=1$. Hence we can assume, without loss of generality, that $G_1$ is nonempty.  Let $n_i$ and $m_i$ denote the number of vertices and edges of $G_i$, respectively, for each $i\in\{1,2\}$. Since $\sigma(G)=1$,
  \[
  \frac{2m_2}{n_2} \le \mu_1(G_2)<\overline d(G)=\frac{2(m_1+m_2)}{n_1+n_2}.
  \]
 This implies that 
 \begin{equation}\label{eq:mu7}
  \frac{2m_2}{n_2}<\frac{2m_1+2m_2}{n_1+n_2}<\frac{2m_1}{n_1}. 
 \end{equation}
 As a consequence of  \eqref{eq:mu7} we have that 
 \[
   \mu_2(G_1)<\frac{2m_1+2m_2}{n_1+n_2} < \frac{2m_1}{n_1} =\overline d(G_1).
 \]
We conclude that  $\sigma(G_1)=1$. Since $\mathcal G$ is closed by taking components, $G_1\in\mathcal G$. Thus, if $G_1$ were co-connected, then $G_1=K_1$, contradicting the assumption that $G_1$ is nonempty. Hence $G_1$ is not co-connected and, by 
Theorem \ref{theorem:disconnected complement}, we have that   $G_1=K_{1,r}$ for some $r\geq 1$. 

From \eqref{eq:mu7} we deduce that 
 \[
  \mu_1(G_2)<\frac{2m_1+2m_2}{n_1+n_2}<\frac{2m_1}{n_1}=\frac{2r}{r+1}<2,
 \]
 and hence, by virtue of Lemma~\ref{lemma:cotainfmu1}, we conclude that $G_2$ must be empty. Then there exists an integer  $s\geq 1$ such that $G_2=sK_1$ and therefore it turns out that  $G= K_{1,r} + sK_1$. The average degree of $G$ is  $\overline d(G)=\frac{2r}{r+1+s}$.  
 If $r=1$, then $\sigma(G)=1$ because the second largest Laplacian eigenvalue of $G$ is $0$. If $r\geq 2$, then,  as the second largest eigenvalue of $G$ is $1$ it follows that $\sigma(G)=1$ if and only if $s<r-1$.
 \end{proof}
A \emph{cograph} is a graph with no induced $P_4$. It is well-known that the only connected and co-connected cograph is $K_1$~\cite{MR0337679}. Hence, Conjecture~\ref{conjecture:3} holds trivially for cographs and, by Theorem~\ref{theorem:+}, Conjecture~\ref{conjecture:1} holds for cographs.
\subsection{Characterizing forests and extended $P_4$-laden graphs with $\sigma=1$}

In this section, we verify Conjecture~\ref{conjecture:1} for forests and extended $P_4$-laden graphs (a common superclass of cographs and split graphs).

A graph class $\mathcal G$ is \emph{monotone} if $G\in\mathcal{G}$ implies that every subgraph of $G$ also belongs to $\mathcal G$. Notice that every monotone graph class is closed by taking components. It can be easily seen that the class of all forests is monotone and thus it is closed by taking components.

\begin{theorem} Conjecture~\ref{conjecture:1} holds for forests.\end{theorem}
\begin{proof} Notice that if $T$ is a connected and co-connected forest, then $T$ is either $K_1$ or a tree with diameter greater than two. By virtue of Theorem~\ref{theorem:+}, it suffices to show that if $T$ is a tree with diameter greater than two, then $\sigma(T)\ge 2$. %
Assume that $T$ is a tree with diameter greater than two. Hence there exists two vertices $v_1$ and $v_2$ such that $d(v_1)\ge d(v_2) \ge 2 > 2-\frac{2}{n} = \overline d(T)$. By Lemma \ref{lemma:second largest eigenvalue}, $\mu_2(T) \geq d_2(T) \geq 2 > \overline d(T)$. Therefore, $\sigma(T) \geq  2$.
\end{proof}

Let $\mathcal H$ be a set of graphs. We use the term \emph{$\mathcal H$-free} for referring to the family of those graphs having no graph in $\mathcal H$ as induced subgraph. If $\mathcal H$ has just one element $H$, we write $H$-free for simplicity. A \emph{split graph}~\cite{MR0505860} is a graph whose vertex set can be partitioned into a clique $C$ and a stable set $S$, such a partition $(C,S)$ of its vertices is called a \emph{split partition}. It is well known that the class of split graphs coincides with the class of $\{2K_2,C_4,C_5\}$-free graphs. A \emph{pseudo-split} graph~\cite{MR1287980} %
is a $\{2K_2,C_4\}$-free graph. Hence the class of pseudo-split graphs is a superclass of split graphs~\cite{MR0505860}. An \emph{extended $P_4$-laden} graph~\cite{MR1420325} is a graph such that every induced subgraph on at most six vertices that contains more than two induced $P_4$'s is a pseudo-split graph. By definition, the class of extended of $P_4$-laden graphs is a superclass of the class of pseudo-split graphs and hence also of split graphs. Moreover, the class of extended $P_4$-laden graphs is a superclass of different superclasses of cographs defined by restricting the number of induced $P_4$'s, including $P_4$-lite graphs~\cite{MR1008453} and $P_4$-tidy graphs~\cite{MR1471348}. A \emph{spider}~\cite{MR1471348} is a graph whose vertex set can be partitioned into three sets $S$, $C$, and $R$, where $S=\{s_1,\dots,s_k\}$ ($k \geq 2$) is a stable set; $C=\{c_1,\dots,c_k\}$ is a clique; $s_i$ is adjacent to $c_j$ if and only if $i=j$ (a \emph{thin spider}), or $s_i$ is adjacent to $c_j$ if and only if 
$i \neq j$ (a \emph{thick spider}); $R$ is allowed to be empty and all the vertices in $R$ are adjacent to all the vertices in $C$ and nonadjacent to all the vertices in $S$. The sets $C$, $S$ and $R$ are called \emph{body},  \emph{legs} and \emph{head}  of the spider, respectively. 
In order to characterize those extended $P_4$-laden graphs with $\sigma(G)=1$, we rely on the following structural result.
\begin{theorem}[{\cite{MR1420325}}]\label{theorem:extendedp4laden}
 Each connected and co-connected extended $P_4$-laden graph $G$ satisfies one of the following assertions:
 \begin{enumerate}%
  \item\label{it:p4laden1} $G$ is isomorphic to $K_1$, $P_5$, $\overline{P_5}$, or $C_5$;

  \item\label{it:p4laden2} $G$ is a spider or arises from a spider by adding a twin to a vertex of the body or the legs; or

  \item\label{it:p4laden3} $G$ is a split graph.
 \end{enumerate}
\end{theorem}

We first obtain the following result which is concerned with spiders or graphs arising from a spider by adding a twin to a vertex of the body or the legs.

\begin{lemma}\label{lemma:spider}
 If $G$ is a spider or a graph that arises from a spider by adding a twin of a vertex of the body or the legs, then $\sigma(G)\geq 2$.
\end{lemma}

 \begin{proof}
 We will prove that $d_2(G)\geq\overline d(G)$. We consider four cases. In each case we denote by $k$ and $n_H$ the number of vertices in the body and in the head  of the corresponding spider, respectively. Recall that $k\ge 2$. 
 \begin{enumerate}%
  \item \emph{Assume that $G$ is a thin spider.} By construction, $d_2(G)=k+n_H$  and $\vert E(G)\vert\ge 2k+n_H$. Hence
  \begin{align*}
   \overline d(G)&\leq\frac{k^2+k+2kn_H+n_H^2-n_H}{2k+n_H}\\
                 &=\frac{2k^2+kn_H+2kn_H+n_H^2}{2k+n_H}-
                   \frac{k^2-k+n_Hk+n_H}{2k+n_H}\\
                 &\le \frac{(k+n_H)(2k+n_H)}{2k+n_H}\\
                 &=d_2(G).
  \end{align*}
 \item \emph{Assume that $G$ arises from a thin spider by adding a twin to a vertex of the body or the leg}. By construction, $d_2(G)\geq k+n_H$ and $\vert E(G)\vert\ge 2k+n_H+1$. Hence
 \begin{align*}
  \overline d(G)&\leq\frac{k^2+k+2kn_H+n_H^2-n_H+2k+2n_H+2}{2k+n_H+1}\\
    &=\frac{2k^2+kn_H+k+2kn_H+n_H^2+n_H}{2k+n_H+1}
    -\frac{k^2-2k-2+kn_H}{2k+n_H+1}\\
    &\le\frac{(k+n_H)(2k+n_H+1)}{2k+n_H+1}=d_2(G).
 \end{align*}
 Notice that the second inequality holds, whenever $k> 2$ or $n_H\neq 0$. However, when $k=2$ and $n_H=0$, it can be verified by inspection that $\overline d(G)\leq d_2(G)$ holds.

 \item \emph{Assume that $G$ is a thick spider.} By construction, $d_2\geq 2(k-1)+n_H$ and $\vert E(G)\vert\ge 2k+n_H$. Hence
 \begin{align*}
 \overline d(G)&\leq\frac{k^2+k+2kn_H+n_H^2-n_H+2k(k-2)}{2k+n_H}\\
      &= \frac{k^2-k+2k^2-2k+2kn_H+n_H^2-n_H}{2k+n_H}\\
      &=\frac{4(k^2-k)+(4k-2)n_H+n_H^2}{2k+n_H}
      -\frac{k^2-k+(2k-1)n_H}{2k+n_H}\\
      &\le \frac{[2(k-1)+n_H](2k+n_H)}{2k+n_H}=d_2(G).
 \end{align*}

 \item \emph{Assume that $G$ arises from a thick spider by adding a twin to a vertex of the body or the leg.} By construction, $d_2\geq 2(k-1)+n_H$ and $d_2\geq 2(k-1)+n_H+1$. Hence
 \begin{align*}
  \overline d(G)&\leq\frac{k^2-k+2k^2-2k+2kn_H+n_H^2-n_H+4(k-1)+2n_H+2}{2k+n_H+1}\\
                &=\frac{3k^2+k-2+(2k+1)n_H+n_H^2}{2k+n_H+1}\\
      &=\frac{4k^2-2k-2+(4k-1)n_H+n_H^2}{2k+n_H+1}
      -\frac{k^2-3k+(2k-2)n_H}{2k+n_H+1}\\
      &\leq \frac{[2(k-1)+n_H](2k+n_H+1)}{2k+n_H+1}=d_2(G).
 \end{align*}
  Notice that the second inequality holds, whenever $k>2$ or $n_H\neq 0$. However, if $k=2$ and $n_H=0$, it can be verified that $\overline d(G)\leq d_2(G)$ holds by inspection.
 \end{enumerate}
 We have shown that in all possible cases, $d_2(G)\geq\overline d(G)$. Hence, by virtue of Lemma~\ref{lemma:second largest eigenvalue},  we conclude that  $\mu_2(G)\geq\overline d(G)$ which means that $\sigma(G)\geq 2$.
 \end{proof}

\begin{theorem}\label{theorem:sigma1split}
     Conjecture~\ref{conjecture:1} holds for split graphs.
 \end{theorem}

 \begin{proof}
 Let $(C,S)$ be a split partition of the graph on $n$ vertices $G$ such that $\vert C\vert =c$ and $\vert S\vert =n-c$. We label the vertices of $G$ so that   $C=\{v_1,\ldots,v_c\}$ and $S=\{v_{c+1},\ldots,v_n\}$.  We can assume, without loss of generality, that $C$ is a maximal clique of $G$ under inclusion and $d_i\ge d_{i+1}$, 
 for each   $ i \in \{1, \ldots,  n-1\}$. 
 
 We claim  that if $G$ is a split graph with $\sigma(G)=1$, then $G$ is isomorphic to $K_{1,r-1}+(n-r)K_1$ for some $r$ such that $2\leq r\leq n$.
 
 In order to prove our claim we assume  that $G$ is nonisomorphic to $K_{1,r-1}+(n-r)K_1$,  for each  $ r \in \{2, \ldots, n\}$ and we will 
 prove that $\sigma(G) \geq 2$. 
 By virtue of Lemma~\ref{lemma:second largest eigenvalue}, it suffices to prove that $d_2 \geq \d$ or equivalently that
 \[
  \sum_{i=3}^n(d_2-d_i)\geq d_1-d_2.
 \]
We will consider two cases. 
\begin{enumerate}
\item \emph{Assume that $d_2\ge c$}. Since $C$ is a maximal clique, $d_2-d_i\ge 1$ for each $i \in \{c+1, \ldots,  n\}$. Hence
 \[
  \sum_{i=3}^n(d_2-d_i)\ge \sum_{i=c+1}^n(d_2-d_i)\ge n-c\ge d_1-d_2.
 \]

\item \emph{Assume that $d_2=c-1$}. %
      Our assumption on $G$ implies that $c > 2$. Moreover, we have that  $d_i\le 1$ for each $i \in \{c+1, \ldots, n\}$. %
      Consequently, $d_2-d_i\ge 1$ for each such $i$ and the reasoning follows as above. 	
\end{enumerate}
 Thus we have proved our claim. 
 In particular, the only connected and co-connected split graph with $\sigma=1$ is $K_1$; i.e., Conjecture~\ref{conjecture:3} holds for split graphs. Therefore, by virtue of Theorem~\ref{theorem:+}, Conjecture~\ref{conjecture:1} holds for split graphs.
 \end{proof}
By combining Theorem~\ref{theorem:extendedp4laden}, Lemma~\ref{lemma:spider}, and Theorem~\ref{theorem:sigma1split}, we obtain the following result.

\begin{theorem}
 Conjecture~\ref{conjecture:1} holds for extended $P_4$-laden graphs.
\end{theorem}

 \begin{proof}
 Let $G$ be a connected and co-connected graph extended $P_4$-laden graph with $\sigma(G)=1$.
 The proof of the theorem will follow by considering different cases depending on which of the assertions of Theorem~\ref{theorem:extendedp4laden} hold. Because of Lemma~\ref{lemma:spider}, $G$ does not satisfy assertion~\ref{it:p4laden2}. If $G$ satisfies assertion~\ref{it:p4laden1}, then $G$ is isomorphic to $K_1$ because $\sigma(C_5)=\sigma(P_5)=\sigma(\overline{P_5})=2$. If $G$ satisfies assertion~\ref{it:p4laden3}, then $G$ is isomorphic to $K_1$ because of Theorem~\ref{theorem:sigma1split}. We conclude that Conjecture~\ref{conjecture:3} holds. Therefore, Theorem~\ref{theorem:+} implies that Conjecture~\ref{conjecture:1} holds for all extended $P_4$-laden graphs $G$.
 \end{proof}

\section*{Acknowledgments}

A.\ Cafure and E.\ Dratman were partially supported by CONICET PIP 112-2013-010-0598. A.\ Cafure, E.\ Dratman, L.N.\ Grippo, and M.D.\ Safe were partially supported by CONICET UNGS-144-20140100027-CO. M.D.\ Safe was partially supported by PGI UNS 24/ZL16. V.\ Trevisan acknowledges the support of  CNPq - Grant 303334/2016-9.

\end{document}